\documentclass[11pt]{article}

\pdfoutput=1

\usepackage{amsmath, amssymb, amsthm, enumerate, fullpage, bm}
\usepackage{verbatim, enumitem, bbm, etoolbox}
\usepackage{latexsym}
\usepackage{amsmath, amsfonts, amssymb, latexsym, mathrsfs,stmaryrd,enumitem}
\usepackage[numbers,sort&compress]{natbib}

\newenvironment{fentail}{\buildrel\textstyle ^{f}\over{\hbox{\vrule height-10pt depth5pt
width-3pt}{\smash\models}}}

\newenvironment{nfentail}{\buildrel\textstyle ^{f}\over{\hbox{\vrule height-10pt depth5pt
width-3pt}{\smash\nvDash}}}

\newenvironment{dotto}{\buildrel\textstyle.\over{\hbox{\vrule height3pt depth0pt width0pt}{\smash\to}}}

\newenvironment{dotv}{\buildrel\textstyle.\over{\hbox{\vrule height3pt depth0pt width0pt}{\smash\vee}}}

\newenvironment{dota}{\buildrel\textstyle.\over{\hbox{\vrule height-1pt depth0pt width0pt}{\smash\wedge}}}

\newenvironment{g}{G\"{o}del~}

\newenvironment{lo}{{\L}ukasiewicz~}

\newenvironment{ds}{\displaystyle}

\newtheorem{thm}{Theorem}[section]
\newtheorem{defn}[thm]{Definition}

\newtheorem{cor}[thm]{Corollary}
\newtheorem{lem}[thm]{Lemma}

\newtheorem{exa}[thm]{Example}
\newtheorem{rem}[thm]{Remark}
\newtheorem{con}[thm]{Conjecture}

\begin{document}	

\bibliographystyle{plain}
	
\author{Seyed Mohammad Amin Khatami
}
\title{Additive G\"{o}del Logic}
\date{2017/11/08}
	
\maketitle

\begin{abstract}
We study an extension of \g propositional logic whose corresponding algebra is an ordered Abelian group. Then we expand the ideas to first-order case of this logic.
\end{abstract}

\section{Introduction}
Extending model-theoretic techniques from classical model theory to other logics is a fashionable trend. The merit of this trend is twofold. Firstly, it can be viewed as a measurement for complexity of semantical aspects of a given logic and, secondly, can be used as an instrumental tool to verify certain fundamental logical questions. Following this, the present paper can been as further development initiated in \cite{khatami2014rational} for studying model-theoretic aspects of extensions of first-order  G\"odel Logic. While in \cite{khatami2014rational}, the first-order G\"odel Logic is enriched by adding countably many nullary logical constants for rational numbers, here we extend it in other way by adding a group structure on the set of truth values. This extension enables us to strengthen considerably the expressive power of the G\"odel Logic. On the other hand, we will see that this strengthening does not prevent us to have nice model-theoretic properties. Therefore, this extension enjoys a balance between the expressive power, on one hand, and nice model-theoretic properties, on the other hand.

The additive \g logic not only involve the \g logic but also it includes the \lo logic. So, this logic can be viewed as a common non-trivial extension of both of \g and \lo logic. Here by non-trivial we mean that this logic is not a boolean logic \cite[Section 4.3]{hajek98}. We
noted that the common extensions of known fuzzy logics are extensively studied by some authors \cite{esteva2009,esteva2007adding,esteva2006,cintula2009distinguished,EGM2011,cintula2003advances,cintula2001l,esteva2001pi,horvcik2004product,esteva2000}. Also in
some papers, basic model-theoretic notions of fuzzy logics are studied \cite{hajek2006theories,dellunde2013elementary,dellunde2012preserving,dellunde2010elementary,novak2003model,benusthenbre08}.

This paper organized as follows. Next section, is devoted to introducing basic syntactical and semantical aspects of additive  G\"odel Logic. In third section, we show that the first-order additive G\"odel Logic satisfies the compactness theorem. The crucial compactness theorem would allow us to develop some model theory for this logic and prove that the additive G\"odel logic satisfies the joint Robinson consistency theorem, and furthermore it is shown that the
class of (ultrametric exhaustive) models of a first-order theory $T$
with respect to elementary substructure forms an abstract elementary class..

\label{intro}
\section{Additive G\"odel Logic }\label{section additive godel logic}
Fuzzy logics are usually enjoy a semantic based on the unite interval of
real numbers. However, the truth value set could be
considered as any linear ordered structure. In this paper we want to work with a fuzzy logic, whose
semantic is based on totally ordered Abelian groups.

The primary logical connectives of propositional additive \g logic, AG, are
\g implication $\to$, \g conjunction $\wedge$,
group conjunction $\oplus$, and group negation $\ominus$ together with the truth-constants $\bot$ and $\mathfrak{e}$.
Some of further connectives are defined as follows.
\begin{eqnarray*}
  0\varphi&:=&\mathfrak{e}\\
  n\varphi&:=&(n-1)\varphi\oplus\varphi\\
  \varphi\vee\psi &:=& ((\varphi\to\psi)\to\psi)\wedge((\psi\to\varphi)\to\varphi)\\
  \neg\varphi &:=& \varphi\to\bot\\
  \varphi\leftrightarrow\psi &:=& (\varphi\to\psi)\wedge(\psi\to\varphi)\\
  \top &:=& \neg\bot \\
\end{eqnarray*}
The semantic of additive \g logic is based on totally ordered Abelian groups.
\begin{defn}\label{gammag}
Let $(G, +, 0_G, \le)$ be a totally ordered Abelian group.
Set $\Gamma_G=G\cup\{\bar{\infty},\bar{\propto}\}$, and let
\begin{itemize}
\item[] $\bar{\infty} + \bar{\propto}=\bar{\propto} + \bar{\infty}=0_G$,
\item[] for all $a\in G$, $a + \bar{\infty}=\bar{\infty} + a=\bar{\infty}$ and
$a + \bar{\propto}=\bar{\propto} + a=\bar{\propto}$,
\item[] $-\bar{\propto}=\bar{\infty}$ and $-\bar{\infty}=\bar{\propto}$.
\end{itemize}
Extend the order $\le$ on $\Gamma_G$ such that $\bar{\propto}$ and $\bar{\infty}$ be the least and largest
elements of $\Gamma_G$.
\end{defn}
For a given totally ordered Abelian group $G$, we consider $\Gamma_G$ as
the set of truth values, whereas $\bar{\propto}$ is the absolute falsity and $\bar{\infty}$
is the absolute truth. Define the following operators on $\Gamma_G$.
\begin{eqnarray*}
a\dotv b&=&\max\{a,b\}\\
a\dota b&=&\min\{a,b\}\\
a\dotto b&=&\left\{\begin{array}{cc}
\bar{\infty}&~~~~a\le b\\
b&~~~~a>b
\end{array}\right.\\
d_{min}(a,b)&=&\left\{\begin{array}{cc}
\min\{a,b\}&~~~~a\ne b\\
\bar{\infty}\hfill&~~~~a=b
\end{array}\right.
\end{eqnarray*}
Now, for a given totally ordered Abelian group $(G,+,0_G,\le)$ and truth-value set
$\Gamma_G=G\cup\{\bar{\infty},\bar{\propto}\}$, a $G$-evaluation is any mapping
$v$ assigning to each propositional variable $p$ an element $e(p)\in \Gamma_G$ which extends
to all propositionals by setting
\begin{itemize}
\item $v(\bot)=\bar{\propto}$,
\item $v(\mathfrak{e})=0_G$,
\item $v(\varphi\oplus\psi)=v(\varphi)+v(\psi)$,
\item $v(\ominus\varphi)=-v(\varphi)$,
\item $v(\varphi\wedge\psi)=v(\varphi)\dota v(\psi)$,
\item $v(\varphi\to\psi)=v(\varphi)\dotto v(\psi)$,
\end{itemize}
and so we have $v(\top)=\bar{\infty}$, $v(\varphi\vee\psi)=v(\varphi)\dotv v(\psi)$, and
$v(\varphi\leftrightarrow\psi)=d_{min}\big(v(\varphi),v(\psi)\big)$.
\begin{exa}
Let $f:(0,1)\to\mathbb{R}$ is defined by $f(x)=\tan(\pi x-\ds\frac{\pi}{2})$.
For each $x,y\in (0,1)$ define
\begin{center}
$x\star y=f^{-1}\big(f(x)+f(y)\big)$.
\end{center}
One could easily verify that $(0,1)_{AG}=\big((0,1), \star, \ds\frac{1}{2}, \le\big)$
is a totally ordered Abelian group. Now, by setting $\bar{\infty}=1$ and $\bar{\propto}=0$ we
have the truth value set $\Gamma_{[0,1]}=(0,1)_{AG}\cup\{0,1\}$ in which for every $a,b\in[0,1]$,
$a\dotto b=\left\{\begin{array}{cc}
1&~~a\le b\\
b&~~a>b
\end{array}\right.$. Thus
$v(\varphi\to\psi)=\left\{\begin{array}{cc}
1&~~v(\varphi)\le v(\psi)\\
b&~~v(\varphi)> v(\psi)
\end{array}\right.$ which is the standard \g implication. Furthermore
$v(\ominus\varphi)=1-v(\varphi)$ which is the standard \lo negation. On the other hand if we set
$\varphi\rightsquigarrow\psi:=\varphi\oplus(\ominus\psi)$, then
$v(\varphi\rightsquigarrow\psi)=v(\psi)\oplus \big(1-v(\psi)\big)$. Now, one can easily see that
the truth function of $\rightsquigarrow$ is a continuous map a bit similar to the \lo implication.
Indeed if $v(\varphi)\le v(\psi)$, then $v(\varphi\rightsquigarrow\psi)\ge\ds\frac{1}{2}$ and the
absolute truth (absolute falsity) of $\rightsquigarrow$ is only take place for
$\bot\rightsquigarrow\top$ ($\top\rightsquigarrow\bot$).
\end{exa}
\begin{exa}
$\Gamma_{\mathbb{R}+}=(\mathbb{R}, +, 0, \le)\cup\{-\infty,\infty\}$ with the usual
meaning of sum and order in real numbers, is the standard truth value set for AG. However
one might use the multiplicative representation
$\Gamma_{\mathbb{R}\times}=(\mathbb{R}^{>0}, ., 1, \le)\cup\{0,\infty\}$ with the usual
meaning of product and order in real numbers.
\end{exa}
We now present the axioms of propositional additive \g logic. The first seven from an axiomatizations for propositional \g logic \cite{hajek98}.
\begin{itemize}[leftmargin=1cm]
\item[(G1)] $(\varphi \rightarrow \psi)\rightarrow ((\psi\rightarrow \chi)\rightarrow(\varphi\rightarrow
\chi))$
\item[(G2)] $(\varphi\wedge\psi)\rightarrow \varphi$
\item[(G3)] $(\varphi\wedge\psi)\rightarrow(\psi\wedge\varphi)$
\item[(G4)] $\varphi\rightarrow(\varphi\wedge\varphi)$
\item[(G5)] $(\varphi\rightarrow(\psi\rightarrow\chi))\leftrightarrow((\varphi\wedge\psi)\rightarrow
\chi)$
\item[(G6)] $((\varphi\rightarrow\psi)\rightarrow\chi)\rightarrow(((\psi\rightarrow\varphi)\rightarrow\chi)\rightarrow\chi)$
\item[(G7)] $\bot\rightarrow\varphi$
\end{itemize}
The next five axioms  of additive \g logic say that $\oplus$, $\ominus$ and $\mathfrak{e}$ behaves such that

\begin{rem}
Writing $G$ multiplicatively together with assuming a similarity
relation leads to obtain a pseud ultrametric. So we choose multiplicative
notion for totally ordered Abelian groups.
\end{rem}

As in classical first-order logic, we work with first-order languages.
Firstly, we introduce an extension of \g logic, named additive \g logic.
\begin{defn}
The first-order additive \g logic, AG$\forall$,
consists of the following logical symbols:
\begin{enumerate}
  \item Logical connectives $\wedge$, $\to$ , $\otimes$, $^{-1}$, $\bar{1}$ and $\bot$.
  \item Quantifiers $\forall$ and $\exists$.
  \item A countable set of variables $\{x_n\}_{n\in\mathbb{N}}$.
\end{enumerate}
First-order languages are defined the same as classical first-order logic and
are considered as non-logical symbols of AG$\forall$. So a language $\tau$ is a
set
\begin{center}
$\tau=\left\{\{(f_i,n_{f_i})\}_{i\in I},\{(P_j,n_{P_j})\}_{j\in J}\right\}$
\end{center}
in which for every $i\in I$, $f_i$ is a function symbol of
arity $n_{f_i}\ge 0$ and for each $j\in J$, $P_j$ is a predicate symbol of arity $n_{P_j}\ge 0$.
A nullary function symbol is commonly called a constant symbol.
\end{defn}
For a given first-order language $\tau$, the usual
definition of $\tau$-terms and (atomic) $\tau$-formulas
are considered. Free and bound variables defined as in
classical first-order logic. A $\tau$-sentence is
a $\tau$-formula without free variable. The set of
$\tau$-formulas and $\tau$-sentences are denoted by
$Form(\tau)$ and $Sent(\tau)$, respectively. When
there is no danger of confusion we may omit the prefix $\tau$
and simply write a term, (atomic) formula or sentence. A theory is a
set of sentences.

Further connectives are defined as follows.
\begin{eqnarray*}
  \varphi^1&:=&\varphi\\
  \varphi^n&:=&\varphi^{n-1}\otimes\varphi\\
  \varphi\vee\psi &:=& ((\varphi\to\psi)\to\psi)\wedge((\psi\to\varphi)\to\varphi)\\
  \neg\varphi &:=& \varphi\to\bot\\
  \varphi\leftrightarrow\psi &:=& (\varphi\to\psi)\wedge(\psi\to\varphi)\\
  \top &:=& \neg\bot \\
  \varphi\Rightarrow\psi &:=& (\psi\to\varphi)\to\psi\\
  \varphi\Rrightarrow\psi &:=& ((\varphi\Rightarrow\psi)\wedge\neg\neg\psi^{-1})\vee(\psi\wedge\neg\neg\varphi^{-1})\\
  \Delta(\varphi)&:=&\neg(\varphi\Rrightarrow T)\\
  \varphi\to_{_L}\psi&:=&\bar{1}\to(\psi\otimes\varphi^{-1})
\end{eqnarray*}

\begin{defn}
For a given language $\tau$, a $\tau$-\emph{structure} $\mathcal{M}$ is a
nonempty set M called the universe of $\mathcal{M}$ together with:
  \begin{itemize}
  \item[a)] a totally ordered Abelian group $(G, *, \le)$ with identity element $1_G$ or the empty set,
  \item[b)] for any n-ary predicate symbol $P$ of $\tau$, a function $P^{\mathcal{M}}:M^n\to\Gamma_G$, while
  for nullary predicate symbol, $P^\mathcal{M}$ is an element of $\Gamma_G$,
  \item[c)] for any n-ary function symbol $f$ of $\tau$, a function $f^{\mathcal{M}}:M^n\to M$, while in
  the case of nullary function symbol, $f^\mathcal{M}$ is an element of $M$.
  \end{itemize}
\end{defn}
We may call $\mathcal{M}$ a $\tau^G$-structure. Sometimes $\mathcal{M}$ is denoted by
$\mathcal{M}=(G,M)$. When there is no fear of confusion, we may omit the underlying language
$\tau$ and the group symbols $G$, and call $\mathcal{M}$ a structure. A structure whose underlying
group is a totally ordered subgroup of $(\mathbb{R}^{>0},.,1)$ is called a standard structure.

For each $\alpha\in\tau$, $\alpha^\mathcal{M}$ is called the
\emph{interpretation} of $\alpha$ in $\mathcal{M}$. The interpretation of
terms defined inductively as follows.
\begin{defn}
For every n-tuple $\bar{x}=x_1, x_2, ..., x_n$ and every term $t(\bar{x})$,the interpretation of $t(\bar{x})$ in
$\mathcal{M}$ is a function $t^\mathcal{M}:M^n\to M$ such that
\begin{enumerate}
\item if $t(\bar{x})=x_i$ for $1\le i\le n$, then $t^\mathcal{M}(\bar{a})=a_i$,
\item if $t(\bar{x})=f(t_1(\bar{x}),...,t_m(\bar{x}))$ then $t^\mathcal{M}(\bar{a}) = f^\mathcal{M}(t_1^\mathcal{M}(\bar{a}),...,t_m^\mathcal{M}(\bar{a}))$.
\end{enumerate}
\end{defn}
Similarly, the interpretation of formulas in structures is defined as follows.
\begin{defn}
The interpretation of a formula $\varphi(\bar{x})$ in a $\tau^G$-structure $\mathcal{M}$
is a function $\varphi^\mathcal{M}:M^n\to\Gamma_G$ which is inductively determined as follows.
\begin{enumerate}
\item $\bot^\mathcal{M}=0$, $\top^\mathcal{M}=\infty$ and $\bar{1}=1_G$.
\item For every n-ary predicate symbol $P$,
\begin{center}
$P^\mathcal{M}(t_1(\bar{a}),...,t_n(\bar{a}))=P^\mathcal{M} (t_1^\mathcal{M}(\bar{a}),...,t_n^\mathcal{M}(\bar{a}))$.
\end{center}
\item $(\varphi \wedge \psi)^\mathcal{M}(\bar{a}) = \varphi^\mathcal{M}(\bar{a}) \dota \psi^\mathcal{M}(\bar{a})$.
\item $(\varphi \to \psi)^\mathcal{M}(\bar{a}) =\varphi^\mathcal{M}(\bar{a})\dotto\psi^\mathcal{M}(\bar{a})$.
\item $(\varphi \otimes\psi)^\mathcal{M}(\bar{a}) = \varphi^\mathcal{M}(\bar{a}) * \psi^\mathcal{M}(\bar{a})$.
\item $(\varphi^{-1})^\mathcal{M}(\bar{a}) = (\varphi^\mathcal{M}(\bar{a}))^{-1}$.
\item if $\varphi(\bar{x})=\forall y\ \psi(y,\bar{x})$ then $\varphi^\mathcal{M}(\bar{a})=\inf_{b\in M}\{\psi^\mathcal{M}(b,\bar{a})\}$.
\item if $\varphi(\bar{x})=\exists y\ \psi(y,\bar{x})$ then $\varphi^\mathcal{M}(\bar{a})=\sup_{b\in M}\{\psi^\mathcal{M}(b,\bar{a})\}$.
\end{enumerate}
The suprema and infima may not be exist. When all suprema and infima
exist in an structure $\mathcal{M}$, we call $\mathcal{M}$ a safe
structure. We assume all structures to be safe hereafter.
One could easily verify that:
\begin{eqnarray*}
(\varphi \vee\psi)^\mathcal{M}(\bar{a})&=&\varphi^\mathcal{M}(\bar{a}) \dotv \psi^\mathcal{M}(\bar{a}),\\
(\varphi \leftrightarrow \psi)^\mathcal{M}(\bar{a})&=&d_{min}(\varphi^\mathcal{M}(\bar{a}),\psi^\mathcal{M}(\bar{a})),\\
(\neg\neg\varphi)^\mathcal{M}(\bar{a})&=&\left\{\begin{array}{cc}
\infty&~~~~~\varphi^\mathcal{M}(\bar{a})>0,\\
0&~~~~~\varphi^\mathcal{M}(\bar{a})=0,
\end{array}\right.\\
(\varphi \Rightarrow \psi)^\mathcal{M}(\bar{a})&=&\left\{\begin{array}{cc}
\infty\hfill&\varphi^\mathcal{M}(\bar{a})<\psi^\mathcal{M}(\bar{a})<\infty,\\
\psi^\mathcal{M}(\bar{a})&otherwise,
\end{array}\right.\\
(\varphi \Rrightarrow \psi)^\mathcal{M}(\bar{a})&=&\left\{\begin{array}{cc}
\infty\hfill&\varphi^\mathcal{M}(\bar{a})<\psi^\mathcal{M}(\bar{a}),\\
0\hfill&\varphi^\mathcal{M}(\bar{a})=\psi^\mathcal{M}(\bar{a})=\infty,\\
\psi^\mathcal{M}(\bar{a})&otherwise,
\end{array}\right.\\
(\Delta(\varphi))^\mathcal{M}(\bar{a})&=&
\left\{\begin{array}{cc}
\infty&~~~~\varphi^\mathcal{M}(\bar{a})=\infty,\\
0&otherwise,
\end{array}\right.\\
(\varphi \to_{_L} \psi)^\mathcal{M}(\bar{a})&=&\left\{\begin{array}{cc}
\infty\hfill&\varphi^\mathcal{M}(\bar{a})\le\psi^\mathcal{M}(\bar{a}),\\
\psi^\mathcal{M}(\bar{a})*(\varphi^\mathcal{M}(\bar{a}))^{-1}&otherwise.
\end{array}\right.
\end{eqnarray*}
\end{defn}
\begin{rem}
The truth functionality of $\to$, $\leftrightarrow$, and $\wedge$ show that
the logic that we work on it is an extension of
\g logic. On the other hand, we have some additional connectives such as
$\otimes$, $\to_{_L}$, and $^{-1}$ whose truth functionalities
acts as the truth functionality of connectives of \lo logic (in the multiplicative notion).
Note that we could define the $\Delta$-Bazz
connective \cite{baaz1996infinite} also.

However the expressive power of the logic is strictly stronger than
\g logic and \lo logic. Observe that as opposed to the \lo logic
we could express $\varphi^\mathcal{M}<\psi^\mathcal{M}$ by
$(\varphi\Rrightarrow\psi)^\mathcal{M}=\infty$ and also
opposed to \g logic (and also \lo logic) we could express $\varphi^\mathcal{M}<\infty$
by $(\varphi\Rrightarrow\top)^\mathcal{M}=\infty$ or
$(\neg\Delta(\varphi))^\mathcal{M}=\infty$.

On the other hand the expressive power is
weaker than the $\L\Pi$ logic \cite{esteva2001pi,cintula2001l,cintula2003advances} as we could not express the
product conjunction (in additive notion).
\end{rem}
The semantical notions of satisfiability, model and entailment are defined as follows.
\begin{defn}
Let $\varphi(\bar{x})$ be a $\tau$-formula, $\psi$ be a $\tau$-sentence, and $T$ be a $\tau$-theory.
\begin{itemize}
\item[(1)] If there is a $\tau^G$-structure $\mathcal{M}$ and $\bar{a}\in M^n$
such that $\varphi^\mathcal{M}(\bar{a})=\infty$, then we call
$\varphi(\bar{x})$ a satisfiable formula. In this case, write
$\mathcal{M}\models\varphi(\bar{a})$ and call $\mathcal{M}=(G,M)$ a model
of $\varphi(\bar{x})$. The class of all models of
$\varphi(\bar{x})$ is denoted by $Mod(\varphi(\bar{x}))$.
\item[(2)] We call $T$ a satisfiable theory if $\cap_{\varphi\in T} Mod(\varphi)\ne\emptyset$.
When $\mathcal{M}\in\cap_{\varphi\in T} Mod(\varphi)$ we say that
$\mathcal{M}$ is a model of $T$ and denote this by
$\mathcal{M}\models T$. The class of all models of $T$
are denoted by $Mod(T)$.
\item[(3)] $T$ is called finitely satisfiable if every finite subset of $T$ has a model.
\item[(4)] $T\models\psi$, if $Mod(T)\subseteq Mod(\varphi)$.
In this case we say that $T$ \emph{entails} $\psi$.
\item[(5)] We write $T\fentail\varphi$ if there exist a finite subset $S$ of $T$ such that $S\models\varphi$.
\end{itemize}
\end{defn}
As in first-order logic the full theory of a $\tau$-structure $\mathcal{M}$ is
\begin{center}
$Th_\tau(\mathcal{M})=\{\varphi: \mathcal{M}\models\varphi, \varphi\in Sent(\tau)\}$.
\end{center}
We may write $Th(\mathcal{M})$ when there is no fear of confusion about the underlying language.
\section{Compactness Theorem}
In this section, using the Henkin construction, we obtain a version of compactness theorem for
additive \g logic.
\begin{defn}
Let $T$ be a $\tau$-theory.
\begin{enumerate}
\item  $T$ is called a \emph{linear complete theory}, if for every $\tau$-sentences $\varphi$ and $\psi$,
either $\varphi\to\psi\in T$ or $\psi\to\varphi\in T$.
\item We say that $T$ is Henkin , if for every $\tau$-formula $\varphi(x)$ that
$T\nfentail\forall x\,\varphi(x)$, there exists some constant $c$ in $\tau$ such that
$T\nfentail\varphi(c)$.
\end{enumerate}
\end{defn}
Bellow, we prove the entailment compactness for AG$\forall$.
Obviously, the entailment compactness implies the usual compactness theorem.
The next theorem is a special case of the entailment compactness
where $T$ is linear complete and Henkin
\begin{thm}\label{entailment compactness henkin}
Let $T$ be a linear complete Henkin $\tau$-theory and $\chi$ be a
$\tau$-sentence. Then $T\models\chi$ if and only if $T\fentail\chi$.
\end{thm}
\begin{proof}
From right to left direction is obvious. For the other direction, let $T\models\chi$ and for
the purpose of contradiction, suppose that $T\nfentail\chi$. Define an equivalence
relation $\sim$ on the set of $\tau$-sentences as follows:
\begin{center}
$\varphi\sim\psi$ iff $T\fentail\varphi\leftrightarrow\psi$.
\end{center}
For every $\tau$-sentence $\varphi$, let $[\varphi]$ be the
equivalence class of $\varphi$ with respect to $\sim$. Let $Lind(T)$
be the set of all equivalence classes of $\sim$. Define the
operation $\star$ on $Lind(T)$ by
\begin{center}
$[\varphi]\star[\psi]=[\varphi\otimes\psi]$.
\end{center}
One can easily verify
that $G_{Lind(T)}=(Lind(T)\setminus\{[\top],[\bot]\},\star)$ is an Abelian group with
identity element $[\bar{1}]$. For example, $\star$ is an associative operator, as if
$\varphi_1, \varphi_2, \varphi_3\in Sent(\tau)$, then
\begin{center}
$\big([\varphi_1]\star[\varphi_2]\big)\star[\varphi_3]=[\varphi_1\otimes\varphi_2]\star[\varphi_3]
=\big[(\varphi_1\otimes\varphi_2)\otimes\varphi_3\big]$,\\
$[\varphi_1]\star\big([\varphi_2]\star[\varphi_3]\big)=[\varphi_1]\star[\varphi_2\otimes\varphi_3]
=\big[\varphi_1\otimes(\varphi_2\otimes\varphi_3)\big]$.
\end{center}
Now, if $[\varphi_1],[\varphi_2],[\varphi_3]\in G_{Lind(T)}$, then linear completeness of $T$ implies that
\[T\fentail\bigwedge_{i=1}^3\big(\neg\neg\varphi_i\wedge\neg\Delta(\varphi_i)\big).\]
Hence, by an easy argument we have
\[T\fentail\big((\varphi_1\otimes\varphi_1)\otimes\varphi_1\big)\leftrightarrow\big(\varphi_1\otimes(\varphi_1\otimes\varphi_1)\big).\]
Furthermore, by defining $\lessdot$ on $Lind(T)$ as
\begin{center}
$[\varphi]\lessdot[\psi]$ iff $T\fentail\psi\to\varphi$,
\end{center}
we make the group $G_{Lind(T)}$ a totally ordered Abelian group
such that $Lind(T)$ is $\Gamma_{G_{Lind(T)}}$. The group
$G_{Lind{T}}$ is called the Lindenbaum group of $T$-equivalence sentences.
Now, let $CM(T)$ be the set of all closed
$\tau$-terms, i.e., terms constructed only by constants symbols of
$\tau$. Construct the $\tau^{G_{Lind(T)}}$-structure $\mathcal{M}=(G_{Lind(T)},CM(T))$
by setting its universe to be $CM(T)$, and for each $n$-ary function
symbol $f\in\tau$ and $n$-ary predicate symbol $P\in\tau$ define
\begin{itemize}
  \item $f^\mathcal{M}:CM(T)^n\to CM(T)$ by $f^\mathcal{M}(t_1, ..., t_n)=f(t_1, ...,
  t_n)$,
  \item $P^\mathcal{M}:CM(T)^n\to Lind(T)$ by $P^\mathcal{M}(t_1, ..., t_n)=[P(t_1, ...,
  t_n)]$.
\end{itemize}
One can easily verify that for each
$\tau$-sentence $\varphi$, $\varphi^\mathcal{M}=[\varphi]$,
$\mathcal{M}\models T$ and $\chi^\mathcal{M}\ne \infty$.
\qed\end{proof}
$(G_{Lind(T)},CM(T))$ is called the canonical model of the theory $T$.
We need the following lemma to prove the entailment compactness in general case.
\begin{lem}\label{linear complete}
Let $T$ be a $\tau$-theory and $\chi$ be a $\tau$-sentence and $T\nfentail\chi$. There exists
a linear complete $\tau$-theory $T'\supseteq T$ such that $T'\nfentail\chi$.
\end{lem}
\begin{proof}
It is easy to see that for every $\tau$-sentences $\varphi$ and $\psi$, either $T\cup\{\varphi\to\psi\}\nfentail\chi$ or $T\cup\{\psi\to\varphi\}\nfentail\chi$.
Now, using Zorn's lemma the desirable linear complete theory established.
\qed\end{proof}
Now, we could prove the entailment compactness in general case.
\begin{thm}\label{entailment compactness}
Let $T$ be a $\tau$-theory and $\chi$ be a $\tau$-sentence.
$T\models\chi$ if and only if $T\fentail\chi$.
\end{thm}
\begin{proof}
We prove the non-trivial direction. Suppose that
$T\nfentail\chi$. We show that there exist a language $\tau'\supseteq\tau$ and a
linear complete Henkin $\tau'$-theory $T'\supseteq T$ such that
$T'\nfentail\chi$.

Let $\chi_0=\chi$, $\tau_0=\tau$, and $T_0=T$.
On the basis of Lemma \ref{linear complete} there is a linear complete theory $\overline{T}_0$
containing $T_0$ such that $\overline{T}_0\nfentail\chi_0$. We extend the language
$\tau_0$ by adding a new nullary predicate symbol $\chi_1$
and new constant symbols $\{c_\varphi: \varphi(x)\in Form(\tau_0)\}$ and let
$\tau_1=\tau\cup\{\chi_1\}\cup\{c_\varphi: \overline{T}_0\nfentail\forall x\,\varphi(x)\}$.
Subsequently put
\begin{center}
$T_1=\overline{T}\cup\{\chi_0\to\chi_1\}\cup\{\varphi(c_\varphi)\to\chi_1: \overline{T}_0\nfentail\forall x\,\varphi(x)\}$.
\end{center}

Now, we show that $T_1\nfentail\chi_1$. To this end, let $U$ be a finite subset of $\overline{T}_0$ and
$S=U\cup\{\chi_0\to\chi_1\}\cup\{\varphi_i(c_{\varphi_i})\to\chi_1: \overline{T}_0\nfentail\forall x\,\varphi_i(x)\}_{i=1}^m$. Since $\overline{T}_0$ is linear complete and $\overline{T}_0\nfentail\chi_0$
and also $\overline{T}_0\nfentail\forall x\,\varphi_i(x)$ for $1\le i\le m$ it follows that
$\overline{T}_0\nfentail\chi_0\vee\big(\bigvee_{i=1}^m\forall x\,\varphi_i(x)\big)$. Hence,
there is a $\tau_0$-structure $\mathcal{M}\models U$ such that
$\max\{\chi_0^\mathcal{M},(\forall x\,\varphi_1(x))^\mathcal{M}, ..., (\forall x\,\varphi_m(x))^\mathcal{M}\}=g<\infty$.
Now, interpreting $\chi_1^\mathcal{M}$ by $g$, making $\mathcal{M}$ as
a $\tau_1$-structure such that $\mathcal{M}\models S$ and $\chi_1^\mathcal{M}=g<\infty$.

By iterating the above construction, we get sequence
$\tau_0\subseteq\tau_1\subseteq ...$ of first-order languages,
$T_0\subseteq T_1\subseteq ...\subseteq T_n\subseteq ...$ of $\tau_n$-theories,
and $\{\chi_n\}_{n=0}^\infty$ of $\tau_n$-sentences such that for each $n\ge0$,
$T_n\nfentail\chi_n$ and $T_{n+1}\models\chi_n\to\chi_{n+1}$.
Set $\tau'=\bigcup_{n\ge 0}\tau_n$ and let $T_\infty=\bigcup_{n\ge 0}T_n$.
Clearly, $T_\infty\nfentail\chi$. Thus, on the basis
of Lemma \ref{linear complete} there exists a linear complete
theory $T'$ containing $T_\infty$ such that $T'\nfentail\chi$. Obviously, $T'$ is a
Henkin $\tau'$-theory. It follows from Theorem \ref{entailment compactness henkin} that
$T'\not\models\chi$. So, $T\not\models\chi$
\qed\end{proof}
The compactness theorem immediately follows from the above theorem.
\begin{cor}(Compactness Theorem)
A theory $T$ is satisfiable if and only if it is finitely
satisfiable.
\end{cor}
\begin{rem}
Note that if $T$ is finitely satisfiable by $G$-models, then it is not
necessarily satisfiable by a $G$-model, while by the above corollary $T$
is satisfiable by a $G'$-model for some totally ordered Abelian group $G'$. To see this,
let $\mathcal{L}=\{\epsilon, \rho\}$ be a relational language consisting of two
nullary predicate symbols. Set,
\begin{center}
$T=\{\bar{1}\Rrightarrow\rho, \epsilon\Rrightarrow\top\}\cup\{\rho^n\Rrightarrow\epsilon\}_{n\in\mathbb{N}}$.
\end{center}
$T$ is finitely satisfiable by standard models, but it has no standard model.
On the other hand, observe that by compactness theorem $T$ is satisfiable.
For example, if we take $G=(\mathbb{R}^{>0})^2$ with
the lexicographical ordering and the componentwise
multiplication, then $T$ has a $G$-model.
\end{rem}
One could naturally ask weather any satisfiable theory has a standard model.
\begin{con}\label{stmodel}
If $T$ is a finite satisfiable theory, then it has a standard model.
\end{con}
\section{Some Model Theory}
In this section, some basic model theoretic concepts of AG$\forall$ is studied.
Various model theoretic definitions such as elementary equivalence, elementary embedding, substructure, and
diagram are studied recently in the context of mathematical fuzzy logics
\cite{cintula2010triangular,hajek2006theories,dellunde2010elementary,dellunde2012preserving,dellunde2013elementary}.

In this paper, we assume that the underlying language $\tau$ contains a binary predicate
symbol which reflects the properties of the equality relation. This assumption
is necessary, since most model theoretic results can not be achieved without the
equality relation.
\subsection{AG$\forall$ with the Equality Relation}\hfill

In the rest of this section, fix a first-order language $\tau_e$ including
a binary predicate symbol $e$. This predicate plays the same role as the equality
relation in classical first-order logic. The essential properties of the equality
relation are the similarity axioms, i.e.,
\begin{center}
$\forall x\,(x=x)$,\\
$\forall x\forall y\,(x=y\to y=x)$,\\
$\forall x\forall y \forall z\,\left((x=y\wedge y=z)\to x=z\right).$
\end{center}
Let $\mathcal{M}$ be a $\tau_e$-structure which models the following similarity axioms.
\begin{center}
$\left\{\forall x\,e(x,x),
\forall x\forall y\,\left(e(x,y)\to e(y,x)\right),
\forall x\forall y \forall z\,\left((e(x,y)\wedge e(y,z))\to e(x,z)\right)\right\}$.
\end{center}
Then, for all $a,b,c\in\mathcal{M}$,
\begin{center}
$e^\mathcal{M}(a,a)=\infty$,\\
$e^\mathcal{M}(a,b)=e^\mathcal{M}(b,a)$,\\
$e^\mathcal{M}(a,b)\ge\min\{e^\mathcal{M}(a,c), e^\mathcal{M}(b,c)\}$.
\end{center}
So, the interpretation of $e^{-1}$ in $\mathcal{M}$ is as like as
a pseudo-ultrametric on the universe of $\mathcal{M}$ (a pseudo-metric in which for all $a,b,c\in M$,
$(e^{-1})^\mathcal{M}(a,b)\le\max\{(e^{-1})^\mathcal{M}(a,c),(e^{-1})^\mathcal{M}(b,c)\}$).
\begin{defn}
Let $\mathcal{M}=(G,M)$ be a $\tau_e$-structure. We call $\mathcal{M}$ An ultrametric structure, whenever
for all $a,b,c\in M$
\begin{itemize}
\item $(e^{-1})^\mathcal{M}(a,b)=0$ if and only if $a=b$,
\item $(e^{-1})^\mathcal{M}(a,b)=(e^{-1})^\mathcal{M}(b,a)$,
\item $(e^{-1})^\mathcal{M}(a,b)\le\max\{(e^{-1})^\mathcal{M}(a,c),(e^{-1})^\mathcal{M}(b,c)\}$.
\end{itemize}
To simplify the notions, for a $\tau_e$-ultrametric structure $\mathcal{M}$ we denote
$e^{-1}$ by $d$.
\end{defn}
\begin{exa} Any first-order structure could be viewed as an ultrametric structure. As we expect,
any ultrametric space $(M,d)$ is an ultrametric structure. Also Any normed field (valued field) is an ultrametric structure.
\end{exa}
\begin{defn}
Let $T$ be a $\tau_e$-theory and $\varphi$ be a $\tau_e$-sentence.
\begin{enumerate}
  \item $T$ is called an $m$-satisfiable theory if there is an ultrametric structure
  $\mathcal{M}\models T$.
  \item We call $T$ finitely $m$-satisfiable whenever every finite subset of
  $T$ has an ultrametric model.
  \item $T\models_m\varphi$ if each ultrametric model of $T$, models $\varphi$.
  \item $T\fentail_m\varphi$ if there is a finite subset $S$ of $T$
  such that $S\models_m\varphi$. Otherwise we write $T\nfentail_m\varphi$.
\end{enumerate}
\end{defn}
The ultrametric version of entailment compactness could be established as follows.
\begin{thm}\label{entailment compactness henkin ultrametric}
Let $T$ be a $\tau_e$-theory and $\chi$ be a
$\tau_e$-sentence. $T\models_m\chi$ if and only if $T\fentail_m\chi$.
\end{thm}
\begin{proof}
The proof is similar to the proof of Theorem \ref{entailment compactness}. By the same way as the proof of
Theorem \ref{entailment compactness} we could assume that $T$ is a linear complete
Henkin $\tau_e$-theory. Let $G_{Lind(T)}$ be the Lindenbaum group of $T$-equivalence sentences
introduced in Theorem \ref{entailment compactness henkin}. Note that here  $Lind(T)$ is the set of all
equivalence classes of the relation $\sim$ on $Sent(\tau)$ which is defined by
\begin{center}
$\varphi\sim\psi$  if and only if $T\fentail_m \varphi\leftrightarrow\psi$.
\end{center}
However, the definition of the order on $Lind(T)$ does not change. So, for $\varphi,\psi\in Sent(\tau)$,
$[\varphi]\lessdot[\psi]$ if and only if $T\fentail\psi\to\varphi$.

Define an equivalence relation on
the set of all closed $\tau_e$-terms as follows.
\begin{center}
$t_1\backsim t_2$ if and only if $T\fentail_m e(t_1,t_2)$
\end{center}
Let $\langle t\rangle$ be equivalence class of $t$ and suppose that
$CM_m(T)$ be the set of equivalence classes of $\backsim$. The canonical ultrametric structure $(G_{Lind(T)},CM_m(T))$
of $T$ is constructed as follows:
\begin{itemize}
\item For each $n$-ary function symbol $f$ define the function $f^\mathcal{M}:M^n\to M$ by $f^\mathcal{M}(\langle t_1\rangle, ..., \langle t_n\rangle)=\langle f(t_1, ..., t_n)\rangle$.
\item For each $n$-ary predicate symbol $P$ define $P^\mathcal{M}:M^n\to \Gamma_{G_{Lind(T)}}$ by $P^\mathcal{M}(\langle t_1\rangle, ..., \langle t_n\rangle)=[P(t_1, ..., t_n)]$.
\end{itemize}
Note that $f^\mathcal{M}$ is well-defined. Indeed, if for $1\le i\le n_f$, $\langle t_i\rangle=\langle t'_i\rangle$, then
$T\fentail_m e(t_i,t'_i)$ and by linear completeness of $T$ we have $T\fentail_m \bigwedge_{i=1}^{n_f}e(t_i,t'_i)$. Hence,
there is a finite subset $S$ of $T$ such that for each ultrametric model $\mathcal{N}\models S$,
$(t_i)^\mathcal{N}=(t'_i)^\mathcal{N}$ for $1\le i\le n_f$. So,
\[f^\mathcal{N}\big((t_1)^\mathcal{N}, ..., (t_{n_f})^\mathcal{N}\big)=
f^\mathcal{N}\big((t'_1)^\mathcal{N}, ..., (t'_{n_f})^\mathcal{N}\big)\]
i.e.,
\[\mathcal{N}\models e\big(f(t_1, ..., t_n),f(t'_1, ..., t'_n)\big).\]
But, then as $\mathcal{N}$ is any arbitrary ultrametric model of $S$ we have
\[T\fentail_m e\big(f(t_1, ..., t_n),f(t'_1, ..., t'_n)\big).\]
A similar argument show that $P^\mathcal{M}$ is well-defined and this complete the proof.
\qed\end{proof}
\begin{cor}\label{compactness ultrametric}
A theory $T$ is finitely $m$-satisfiable if and only if it is $m$-satisfiable.
\end{cor}
\subsection{Basic Notions of Model Theory}\hfill

The definition of elementary equivalent models in classical first-order logic is based on
satisfactory of the same sentences by models. In the case of many-valued logic
the same definition could be chosen.
\begin{defn}
Let $\mathcal{M} = (G,M)$ and $\mathcal{N} = (H,N)$ be two $\tau$-structures.
\begin{enumerate}
\item $\mathcal{M}$ and $\mathcal{N}$ are elementary equivalent,
$\mathcal{M}\equiv\mathcal{N}$, if $Th(\mathcal{M})=Th(\mathcal{N})$.
\item If $A\subseteq M,N$ and $\tau(A)$ be the expansion of $\tau$ by adding some new
constant symbols $c_a$ for each $a\in A$, then $\mathcal{M}$ and $\mathcal{N}$
can be viewed naturally as $\tau(A)$-structures. We say that
$\mathcal{M}$ and $\mathcal{N}$ are elementary equivalent over $A$,
$\mathcal{M}\equiv_{_A}\mathcal{N}$, if $Th_{{\tau(A)}}(\mathcal{M})=Th_{{\tau(A)}}(\mathcal{N})$.
\end{enumerate}
\end{defn}
A structure whose underlying group does not contain any unnecessary element is called an
exhaustive structure. This notion firstly appeared in \cite{hajek2005}.
\begin{defn}
For a $\tau$-structure $\mathcal{M}=(G,M)$ let $Gr(\mathcal{M})$ or $Gr((G,M))$
be the ordered subgroup of truth values of all $\tau$-formulas, i.e.,
\begin{center}
$Gr(\mathcal{M})=\{\varphi^\mathcal{M}(\bar{a}): \varphi\in Form(\tau), \bar{a}\subseteq M\}\setminus\{0,\infty\}$.
\end{center}
$\mathcal{M}=(G,M)$ is called an exhaustive structure if $G=Gr(\mathcal{M})$.
\end{defn}
The definition of elementary embedding is
based on the equality of truth values of formulas \cite{hajek2006theories,benusthenbre08}.
for example if $\mathcal{M}$ and $\mathcal{N}$ are two $\tau$-structures
with the same set of truth values, then $\mathcal{M}$ is elementary embedded in
$\mathcal{N}$ if there is an injection $h:M\to N$ such that
\begin{center}
$\varphi^\mathcal{M}(a_1, a_2, ..., a_n)=\varphi^\mathcal{N}(h(a_1), h(a_2), ..., h(a_n))$
\end{center}
For additive \g logic, we give more suitable definition.
\begin{defn}\label{elementary embedding}
Let $\mathcal{M} = (G,M)$ and $\mathcal{N} = (H,N)$ be $\tau$-structures.
We say that $\mathcal{M}$ is elementary embedded in $\mathcal{N}$, if there are an
injection $h:M\to N$ and a strict order preserving group homeomorphism $T:G\to H$ such that:
\begin{itemize}
  \item $h(f^\mathcal{M}(a_1, ..., a_{n_f}))=f^\mathcal{N}(h(a_1), ..., h(a_{n_f}))$, for all function
  symbols $f\in\tau$ and $\bar{a}\in M^{n_f}$,
  \item $T\big(\varphi^\mathcal{M}(a_1, a_2, ..., a_n)\big)=
  \varphi^\mathcal{N}(h(a_1), h(a_2), ..., h(a_n))$,
  for all $\varphi\in Form(\tau)$ and $\bar{a}\subseteq M$.
\end{itemize}
We call $(h,T):\mathcal{M}\hookrightarrow_\tau\mathcal{N}$ an
elementary embedding from $\mathcal{M}$ into $\mathcal{N}$.
$\mathcal{M}$ and $\mathcal{N}$ are called
isomorphic, $\mathcal{M}\cong\mathcal{N}$, if $T$ is a group isomorphism and
there are two elementary embeddings $(h,T):\mathcal{M}\hookrightarrow_\tau\mathcal{N}$
and $(j,T^{-1}):\mathcal{N}\hookrightarrow_\tau\mathcal{M}$. Obviously, in this
case $h$ is a one-to-one correspondence and we call $(h,T)$ an isomorphism.
\end{defn}
Clearly, the isomorphism relation between $\tau$-structures
is an equivalence relation.
\begin{lem}\label{strong elementary}
Let $\mathcal{M}$ and $\mathcal{N}$ be exhaustive structures and
there is an injection $h:M\to N$ such that
\begin{itemize}
  \item $h(f^\mathcal{M}(a_1, ..., a_{n_f}))=f^\mathcal{N}(h(a_1), ..., h(a_{n_f}))$, for all function
  symbols $f\in\tau$ and $\bar{a}\in M^{n_f}$,
  \item $\mathcal{M}\models\varphi(a_1, a_2, ..., a_n)$ if and only if
  $\mathcal{N}\models\varphi(h(a_1), h(a_2), ..., h(a_n))\big)$,
  for all $\varphi\in Form(\tau)$ and $\bar{a}\subseteq M$.
\end{itemize}
There is a strict order preserving group homeomorphism
$I_{_\mathcal{MN}}:Gr(\mathcal{M})\to Gr(\mathcal{N})$ such that $(h,I_{_\mathcal{MN}})$ is an
elementary embedding from $\mathcal{M}$ into $\mathcal{N}$.
\end{lem}
\begin{proof}
Obviously, $I_{_\mathcal{MN}}\big(\varphi^\mathcal{M}(a_1, ...,a_n)\big)=\varphi^\mathcal{N}(h(a_1), ..., h(a_n))$
does the job.
\qed\end{proof}
\begin{rem}
If $\mathcal{M}$ and $\mathcal{N}$ are exhaustive ultrametric $\tau_e$-structures and there
exist a function $h:M\to N$ such that
\begin{quote}
$\mathcal{M}\models\varphi(a_1, a_2, ..., a_n)$ if and only if
$\mathcal{N}\models\varphi(h(a_1), h(a_2), ..., h(a_n))$,
for all $\varphi\in Form(\tau)$ and $\bar{a}\subseteq M$,
\end{quote}
then one could easily see that $(h,I_{_\mathcal{MN}})$ is an elementary embedding.
\end{rem}
One of the nice properties of model theory of first-order logic is
"amalgamating many structures into one structure". To study this property in
additive \g logic, as in classical first-order logic, we need the method of diagram.
\begin{defn}
Let $\mathcal{M}=(G,M)$ be a $\tau$-structure. The \emph{elementary diagram}
of $\mathcal{M}$ is
\begin{center}
$ediag_\tau(\mathcal{M})=Th_{\tau(M)}(\mathcal{M})$.
\end{center}
We may write $ediag(\mathcal{M})$ when there is no danger of confusion
about the underlying language.
\end{defn}
An important property of elementary diagram in classical first-order logic
is describing the structure, i.e., if $\mathcal{M}$ be a $\tau$-structure and $\mathcal{N}$ be
a $\tau(M)$-structure such that $\mathcal{N}\models
ediag(\mathcal{M})$, then there is an elementary embedding $j:\mathcal{M}\hookrightarrow_\tau\mathcal{N}$.

Below, we show that the elementary diagram of an exhaustive ultrametric structure, fully describe
the structure.
\begin{lem}\label{tozih}
Let $\mathcal{M}=(G,M)$ be an exhaustive ultrametric $\tau_e$-structure. Suppose
for some exhaustive ultrametric $\tau_e(M)$-structure $\mathcal{N}=(H,N)$,
$\mathcal{N}\models ediag(\mathcal{M})$. Then, there is a $\tau_e$ elementary embedding
from $\mathcal{M}$ into $\mathcal{N}$.
\end{lem}
\begin{proof}
Define $j:M\to N$ by $j(m)=m^\mathcal{N}$. Obviously, $j$ is an injection. Indeed, if
$a$ and $b$ are two distinct element of $M$, then
$(\bot\Rrightarrow d(a,b))\in ediag(\mathcal{M})$. Thus,
$\mathcal{N}\models \bot\Rrightarrow d(a,b)$. So, $d^\mathcal{N}(a^\mathcal{N},b^\mathcal{N})>0$,
i.e., $j(a)\ne j(b)$.

On the other hand, if for some n-ary function symbol $f$ and element $b\in M$,
$f^\mathcal{M}(a_1, ..., a_n)=b$, then $d(f(a_1, ..., a_n),b)\in ediag(\mathcal{M})$.
So, $\mathcal{N}\models d(f(a_1, ..., a_n),b)$, that is
\begin{center}
$f^\mathcal{N}(j(a_1), ..., j(a_n))=f^\mathcal{N}(a_1^\mathcal{N}, ..., a_n^\mathcal{N})
=b^\mathcal{N}=j(b)=j(f^\mathcal{M}(a_1, ..., a_n))$.
\end{center}
Furthermore, if $\mathcal{M}\models\varphi(a_1, ..., a_n)$ for a $\tau_e$-formula
$\varphi(x_1, ..., x_n)$ and $\bar{a}\in M^n$, then $\varphi(a_1, ..., a_n)\in ediag(\mathcal{M})$. So,
$\mathcal{N}\models\varphi(j(a_1), ..., j(a_n))$. Conversely, if
$\mathcal{N}\models\varphi(j(a_1), ..., j(a_n))$ for a $\tau_e$-formula
$\varphi(x_1, ..., x_n)$ and $\bar{a}\in M^n$, then
$\mathcal{M}\models\varphi(a_1, ..., a_n)$, since otherwise
$\neg\Delta\big(\varphi(a_1, ..., a_n)\big)\in ediag(\mathcal{M})$. But,
this contradicts with $\mathcal{N}\models\varphi(j(a_1), ..., j(a_n))$.
Now, by Lemma \ref{strong elementary} $(j,I_{_\mathcal{MN}})$ is
the desirable elementary embedding.
\qed\end{proof}
Now, we prove elementary amalgamation over ultrametric structures.
\begin{thm}\label{amalgamation th}
Let $\mathcal{A}=(G_A,A)$, $\mathcal{B}=(G_B,B)$ and $\mathcal{M}=(G_M,M)$ be three
exhaustive ultrametric $\tau_e$-structures. Suppose also,
$(j,I_{_\mathcal{MA}}):\mathcal{M}\hookrightarrow_{\tau_e}\mathcal{A}$ and
$(k,I_{_\mathcal{MB}}):\mathcal{M}\hookrightarrow_{\tau_e}\mathcal{B}$ are elementary embeddings.
Then, there are exhaustive ultrametric $\tau_e$-structure $\mathcal{N}=(G_N,N)$ and
elementary embeddings $(j_1,I_{_\mathcal{AN}}):\mathcal{A}\hookrightarrow_{\tau_e}\mathcal{N}$ and
$(k_1,I_{_\mathcal{BN}}):\mathcal{B}\hookrightarrow_{\tau_e}\mathcal{N}$ such that $j_1\circ j=k_1\circ k$.
\end{thm}
\begin{proof}
Let $\tau_A=\tau_e(M)\cup\{c_a: a\in A\setminus j(M)\}$,
$\tau_B=\tau_e(M)\cup\{c_b: b\in B\setminus j(M)\}$ and
$\tau'=\tau_A\cup\tau_B$. Without loss of
generality, we may assume that
$\tau_A\cap\tau_B=\tau_e(M)$. One can naturally
interpret the new constants $c_a$ and $c_m$, for $a\in A$ and $m\in M$
inside the ultrametric $\tau_A$-structure $\mathcal{A}$ by $a$ and
$j(m)$, respectively. Similarly, for $b\in A$ and $m\in M$
interpret $c_b$ and $c_m$ inside $\mathcal{B}$ by $b$ and $k(m)$, respectively.
We want to show that
$ediag(\mathcal{A})\cup ediag(\mathcal{B})$ is an $m$-satisfiable
$\tau'$-theory.

For a given $\varphi(c_{a_1}, ..., c_{a_i}, c_{m_1}, ...,c_{m_j})\in
ediag(\mathcal{A})$, we have
\begin{center}
$\varphi^\mathcal{A}(c_{a_1}, ..., c_{a_i}, c_{m_1}, ...,c_{m_j})=\infty$, and
$(\varphi^{-1})^\mathcal{A}(c_{a_1}, ..., c_{a_i}, c_{m_1}, ...,c_{m_j})=0$.
\end{center}
Thus,
\begin{center}
$\mathcal{A}\models\exists\bar{x}\left
(\varphi(\bar{x}, c_{m_1}, ..., c_{m_j})
\wedge(\varphi^{-1}(\bar{x}, c_{m_1}, ..., c_{m_j})\to\bot)\right)$.
\end{center}
Now, since $j:\mathcal{M}\hookrightarrow\mathcal{A}$ and
$k:\mathcal{M}\hookrightarrow\mathcal{B}$ are elementary embeddings, we have
\begin{center}
$\mathcal{B}\models\exists\bar{x}\left
(\varphi(\bar{x}, c_{m_1}, ..., c_{m_j})
\wedge(\varphi^{-1}(\bar{x}, c_{m_1}, ..., c_{m_j})\to\bot)\right)$.
\end{center}
Hence,
$\sup_{\bar{b}\in B^i}\Big(\varphi^\mathcal{B}(\bar{b}, c_{m_1}, ..., c_{m_j})
\dota\left((\varphi^{-1})^\mathcal{B}(\bar{b}, c_{m_1}, ..., c_{m_j})\dotto 0\right)\Big)=\infty$,
i.e., for any $g\in G_B$
there exists an $i$-tuple $\bar{b}\in B^i$ such that
\begin{center}
$\varphi^\mathcal{B}(\bar{b}, c_{m_1}, ..., c_{m_j})\ge g$ and
$\left((\varphi^{-1})^\mathcal{B}(\bar{b}, c_{m_1}, ..., c_{m_j})\dotto 0\right)\ge g$.
\end{center}
So, for some $i$-tuple $\bar{b}\in B^i$, $(\varphi^{-1})^\mathcal{B}(\bar{b}, c_{m_1}, ..., c_{m_j})=0$.
Whence, by definition of $\infty$,
\begin{center}
$\varphi^\mathcal{B}(\bar{b}, c_{m_1}, ..., c_{m_j})=\infty$.
\end{center}
Thus,
\begin{center}
$\mathcal{B}\models ediag(\mathcal{B})\cup\{\varphi(c_{a_1}, ..., c_{a_i}, c_{m_1}, ...,c_{m_j})\}$
\end{center}
where $\varphi(c_{a_1}, ..., c_{a_i}, c_{m_1}, ...,c_{m_j})$ is an arbitrary element of $ediag(\mathcal{A})$.
A similar argument shows that $ediag(\mathcal{B})\cup ediag(\mathcal{A})$ is finitely $m$-satisfiable. So,
by compactness theorem
$ediag(\mathcal{B})\cup ediag(\mathcal{A})$ is $m$-satisfiable. Now,
any exhaustive ultrametric model
$\mathcal{N}=(G_N,N)\models ediag(\mathcal{A})\cup ediag(\mathcal{B})$
fulfills the requirement.
\qed\end{proof}
\begin{defn}
A $\tau^G$-structure $\mathcal{M}$ is called an elementary substructure of a
$\tau^H$-structure $\mathcal{N}$ (or $\mathcal{N}$ is an elementary
extension of $\mathcal{M}$) if $M\subseteq N$ and the inclusion map from
$M$ into $N$ together with $I_{_\mathcal{MN}}$ be an elementary
embedding. We denote this by $\mathcal{M}\prec\mathcal{N}$.
\end{defn}
Bellow, we see that the class of exhaustive structures is closed
under the union of elementary chains.
\begin{thm}\label{union of chain}
Let $\{\mathcal{M}_i\}_{i=0}^\infty$ be a sequence of exhaustive
$\tau$-structures such that $\mathcal{M}_i\prec \mathcal{M}_{i+1}$
for each $i\ge 1$. There exists a unique $\tau$-structure $\mathcal{M}$
with the underlying universe $M=\cup_{i=1}^\infty M_i$ such that $\mathcal{M}_i\prec\mathcal{M}$
for each $i\ge 1$.
\end{thm}
\begin{proof}
Let $\tau_1=\tau\cup\{c_m: m\in M_1\}$. For each $i\ge 2$ set
$\tau_i=\tau_{i-1}\cup\{c_m: m\in M_i\setminus M_{i-1}\}$ and put
$\tau_\infty=\cup_{i=1}^\infty\tau_i$. Consider
$\mathcal{M}_i$ as a $\tau_i$-structure by interpreting each $c_m\in\tau_i$
by $m$, itself. Fix $E_i=Th_{\tau_i}(\mathcal{M}_i)$.

Obviously, $\Sigma=\cup_{i=1}^\infty E_i$ is finitely satisfiable. Thus,
by the compactness theorem, $\Sigma$ is satisfiable.
We show that there is an exhaustive model $\mathcal{M}$ of $\Sigma$ such that
its underlying universe is $\cup_{i=1}^\infty M_i$ and for each $i\ge 1$,
$\mathcal{M}_i\prec\mathcal{M}$. To this end, we prove the followings:
\begin{enumerate}
\item $\Sigma$ is a linear complete Henkin theory.
\item The underlying universe of the canonical model
$(G_{Lind(\Sigma)},CM(\Sigma))$ of $\Sigma$ is identical to $\cup_{i=1}^\infty M_i$.
\end{enumerate}

The linear completeness is obvious. Now,
for every $\tau_\infty$-formula $\varphi(x,\bar{c})$, assume that $n_\varphi$ be
the least natural number such that $\varphi(x,\bar{c})\in Form(\tau_{n_\varphi})$.
If $\Sigma\nfentail \forall x\,\varphi(x,\bar{c})$ then
$\forall x\,\varphi(x,\bar{c})\notin\Sigma$ and consequently
$\forall x\,\varphi(x,\bar{c})\notin E_{n_\varphi}$. Thus
$(\forall x\,\varphi(x,\bar{c}))^{\mathcal{M}_{n_\varphi}}<\infty$ which
implies that there exists an element $b\in M_{n_\varphi}$ such that
$\varphi^{\mathcal{M}_{n_\varphi}}(b,\bar{c}^{\mathcal{M}_{n_\varphi}})<\infty$.
This implies, $\neg\Delta\big(\varphi(c_b,\bar{c})\big)\in E_{n_\varphi}\subseteq\Sigma$ and
therefore $\Sigma\nfentail \varphi(c_b,\bar{c})$. It follows
that $\Sigma$ is Henkin.

On the other hand, in the light of
Theorem \ref{entailment compactness henkin}
the underlying universe of the canonical model of $\Sigma$ is the set of
closed $\tau_\infty$-terms, which can be easily seen
that it is identical to $\cup_{i=1}^\infty M_i$.

Now, as for each $j\ge 1$, $\mathcal{M}_j$ is exhaustive, the function
$T: Gr(\mathcal{M}_j)\to G_{Lind(\Sigma)}$ defined by
$T(\varphi^{\mathcal{M}_j}(m_1, ..., m_n))=[\varphi(c_{m_1}, ..., c_{m_n}]$ is
a well-defined strict order preserving group homeomorphism . Thus, if
$i$ is the inclusion map from $M_j$ into $\cup_{i=1}^\infty M_i$, then
$(i,T)$ is an elementary embedding from $\mathcal{M}_j$ into $\mathcal{M}$, that is
$\mathcal{M}_j\prec\mathcal{M}$.

Finally, if $\mathcal{P}$ is another exhaustive $\tau$-structure with the
same underlying universe $\cup_{i=1}^\infty M_i$ such that
for each $i\ge 1$, $\mathcal{M}_i\prec\mathcal{P}$, then a
straightforward argument as above paragraph show that $\mathcal{P}\cong\mathcal{M}$.
\qed\end{proof}
The model $(G_{Lind(\Sigma)},\cup_{i=1}^\infty M_i)$ is denoted by $\bigcup_{i=1}^\infty \mathcal{M}_i$.
\begin{lem}
Let $\mathcal{M}_1\prec\mathcal{M}_2\prec...$ be a sequence of exhaustive $\tau$-structures,
$\mathcal{N}=(N,G_\mathcal{N})$ be an exhaustive $\tau$-structure, and
$\mathcal{M}_i\prec\mathcal{N}$, for all $i\ge 1$. Then,
$\bigcup_{i=1}^\infty\mathcal{M}_i\prec\mathcal{N}$.
\end{lem}
\begin{proof}
Assume that $i$ is the inclusion map. For each $j\ge 1$, let
$(i,T_j)$ be an elementary embedding from $\mathcal{M}_i$
into $\mathcal{N}$. Define $T:G_{Lind(\Sigma)}\to G_\mathcal{N}$ by
$T\big([\varphi(c_{m_1}, ..., c_{m_n})]\big)=T_{n_\varphi}(m_1, ..., m_n)$ where
$n_\varphi$ is introduced in Theorem \ref{union of chain}.
One could easily verify that $T$ is a well-defined strict order preserving
group homeomorphism. Now $(i,T)$ is an elementary embedding from
$\bigcup_{i=1}^\infty\mathcal{M}_i$ into $\mathcal{N}$.
\qed\end{proof}


\bibliographystyle{spmpsci}      

\bibliography{amin}

%
%

\end{document}